\documentclass[12pt]{article}
\usepackage[utf8]{inputenc}
\usepackage[a4paper, total={6in, 8in}, left=1.134in]{geometry}
\usepackage{amsmath}
\usepackage{amssymb}
\usepackage{amsthm}
\usepackage{hyperref}
\usepackage{tikz-cd}
\hypersetup{colorlinks=true,linkcolor=black,citecolor=black}

\newtheorem{theorem}{Theorem}[section]
\newtheorem{lemma}[theorem]{Lemma}
\newtheorem{proposition}[theorem]{Proposition}
\newtheorem{corollary}[theorem]{Corollary}
\newtheorem{definition}[theorem]{Definition}
\newtheorem{remark}[theorem]{Remark}

\newtheorem{question}[theorem]{Question}
\numberwithin{equation}{section}

\newcommand{\R}{\mathbb{R}}

\newcommand{\Q}{\mathbb{Q}}
\newcommand{\C}{\mathbb{C}}
\newcommand{\N}{\mathbb{N}}
\renewcommand{\P}{\mathbb{P}}

\newcommand{\OO}{\mathcal{O}}
\newcommand{\T}{\mathcal{T}}

\renewcommand{\d}{\partial}
\newcommand{\dbar}{\overline{\partial}}

\newcommand{\e}{\varepsilon}
\newcommand{\II}{\mathrm{II}}

\DeclareMathOperator{\Ric}{Ric}
\DeclareMathOperator{\Rm}{Rm}
\DeclareMathOperator{\tr}{tr}

\DeclareMathOperator{\Kod}{Kod}
\DeclareMathOperator{\ch}{ch}
\DeclareMathOperator{\td}{td}
\DeclareMathOperator{\Vol}{Vol}

\DeclareMathOperator{\Imag}{Im}

\DeclareMathOperator{\id}{id}

\DeclareMathOperator{\rank}{rank}
\DeclareMathOperator{\codim}{codim}

\newcommand{\oo}[1]{\overline{#1}}
\newcommand{\parabolic}[1]{\left(\frac{\d}{\d t} - \Delta_{#1}\right)}

\newcommand{\norm}[1]{\left\|#1\right\|}
\newcommand{\ip}[1]{\left\langle#1\right\rangle}

\renewcommand{\div}{\textnormal{div}}

\begin{document}

\title{On Miyaoka-Yau Inequalities and Weil-Petersson Metrics}

\author{Alexander Bednarek\\The University of Sydney\\ \texttt{alexander.bednarek@sydney.edu.au}}

\date{\today}
\maketitle

\begin{abstract}
\noindent We use the K\"ahler-Ricci flow to give alternate proofs of several known Miyaoka-Yau inequalities and slope semi-stabilities, in particular, for the case of compact K\"ahler manifolds with semi-ample canonical line bundles and K-semistable Fano manifolds. Moreover, when the canonical line bundle is semi-ample, and the Kodaira dimension is $n-1$, we prove the Miyaoka-Yau quantity is equal to the intersection of the Weil-Petersson metric and the Fubini-Study metric on the canonical model. Consequently, equality holds in the Miyaoka-Yau inequality if and only if the pluricanonical map is a holomorphic fibre bundle.
\end{abstract}

\section{Introduction}

\noindent The classical, \cite{M77,Y77}, Miyaoka-Yau inequality states that for a compact K\"ahler-Einstein manifold, $(X,\omega_{KE})$,
\begin{equation*}
    (2(n+1)c_2(X) - n c_1(X)^2 ) \cdot [\omega_{KE}]^{n-2} \geq 0
\end{equation*}
and equality holds if and only if $\omega_{KE}$ has constant holomorphic bisectional curvature. Then, by the Uniformization Theorem, the universal cover $(\tilde{X}, \tilde{\omega}_{KE})$ is isometric, up to scaling, to either $(\P^n, \omega_{FS})$, $(\C^n, \omega_{Euc})$ or $(\mathbb{B}^n, \omega_{Poin})$, depending on the sign of the Einstein constant, \cite{T00}, where $\omega_{FS}, \omega_{Euc}, \omega_{Poin}$ denote the Fubini-Study, Euclidean, and Poincare metrics, respectively.\\

\noindent In the past decade there has been much effort directed towards extending the Miyaoka-Yau inequality to settings which do not admit a K\"ahler-Einstein metric and understanding geometrically when equality holds. It has been extended to smooth, projective minimal models of general type by Zhang, \cite{Z09}, followed by projective minimal models of general type with klt singularities in \cite{GKPT19} and then to normal, projective minimal models with terminal singularities in \cite{GT22}. Then Liu, \cite{L23}, proved a Miyaoka-Yau inequality for any smooth K\"ahler manifold with nef canonical line bundle, and, most recently, there has been an extension to compact K\"ahler spaces klt singularities and nef canonical sheaf, \cite{ZZZ25}.\\

\noindent For Fano manifolds, Song-Wang, \cite{SW16}, proved that a Miyaoka-Yau inequality holds when a Ricci lower bound condition is satisfied, and Li, \cite{L21}, proved the inequality holds when the manifold is K-semistable. Most recently, Hisamoto, \cite{H24}, proved for projective manifolds with nef anti-canonical line bundle that if the limit superior of the $\delta$-invariant is strictly greater than $1$, the Miyaoka-Yau inequality holds again, and in \cite{IJZ25} there is an extension to compact K\"ahler varieties with klt singularities and nef anti-canonical divisor.\\

\noindent We prove the following inequalities using the K\"ahler-Ricci flow. The first two inequalities are contained in \cite{L23}, \cite{L21}, respectively, but were proven using other methods.

\begin{theorem}
Let $X$ be a compact K\"ahler manifold. If the canonical line bundle, $K_X$, is semi-ample then
\begin{equation*}
    (2(n+1) c_2(X) - n c_1(X)^2) \cdot [\omega_0]^{n-i-2} \cdot (-2 \pi c_1(X))^i \geq 0,
\end{equation*}
for all K\"ahler classes $0 < [\omega_0] \in H^{1,1}(X,\R)$, where $i = \min\{\Kod(X), n-2\}$. If $X$ is a K-semistable Fano manifold then
\begin{equation*}
   (2(n+1) c_2(X) - n c_1(X)^2) \cdot (2 \pi c_1(X))^{n-2} \geq 0.
\end{equation*}
If $X$ is a Fano manifold, $V$ is a non-trivial holomorphic vector field whose imaginary part, $\Imag V$, induces an $S^1$-action on $X$, and the $V$-modified Mabuchi K-energy is bounded below on $2 \pi c_1(X)$ then, for some constant $C_V > 0$ defined by (\ref{divergence_norm}), we have
\begin{equation*}
    (2(n+1) c_2(X) - n c_1(X)^2) \cdot c_1(X)^{n-2} \geq -\frac{C_V(n+2)}{n(n-1)} c_1(X)^n.
\end{equation*}
\end{theorem}

\noindent Additionally, the Miyaoka-Yau inequality frequently gives rise to slope-stability results of certain holomorphic vector bundles due to similar proof methodologies. As elucidated by the Donaldson-Uhlenbeck-Yau theorem, \cite{UY86}, one reason that the notion of slope stability is of interest is due to its strong ties to the existence of canonical metrics. We are able to prove the following result regarding slope stability. 

\begin{theorem}
Let $X$ be a compact K\"ahler manifold. If the canonical line bundle, $K_X$, is nef then for any K\"ahler class, $[\omega_0]$, the holomorphic tangent bundle, $T^{1,0} X$, is $([\omega_0],K_X)$-slope-semistable. If $X$ is a K-semistable Fano manifold then the holomorphic tangent bundle, $T^{1,0} X$, is $-K_X$-slope-semistable. 
\end{theorem}

\noindent Note that the numerical dimension of $K_X$ can be arbitrary due to our normalization of the slope, Definition \ref{normalized_slope}, while the fact that K-semistability implies $-K_X$-slope-semistability of $T^{1,0} X$ for Fano manifolds has already been known, see \cite{L21,H24}.\\

\noindent Finally, we prove the following relation between the Weil-Petersson metric and the Miyaoka-Yau quantity.

\begin{theorem}
Let $X$ be a compact K\"ahler manifold with semi-ample canonical line bundle, $K_X$, smooth canonical model, $X_{can}$, and Kodaira dimension, $\Kod(X) = n-1$. Then
\begin{equation*}
    (2(n+1) c_2(X) - n c_1(X)^2) \cdot (-2 \pi c_1(X))^{n-2} = \frac{12(n+1)}{\pi} [\omega_{WP}] \cdot [\eta_{FS}]^{n-2}
\end{equation*}
where $\eta_{FS}$ is a scalar multiple of the Fubini-Study metric and $\omega_{WP}$ is the Weil-Petersson metric.
\end{theorem}

\noindent Due to the semi-positivity of the Weil-Petersson metric, this provides a new proof of the Miyaoka-Yau inequality. We briefly comment that the Weil-Petersson metric is defined as the $L^2$-inner product of harmonic representatives of the Kodaira-Spencer map, and measures the infinitesimal variation of the complex structures among the fibres, \cite{T87, ST12,T19}. As such, we get the following new geometric characterization of equality in the Miyaoka-Yau inequality.

\begin{corollary}
Let $X$ be a compact K\"ahler manifold with semi-ample canonical line bundle, $K_X$, smooth canonical model, $X_{can}$, and Kodaira dimension, $\Kod(X) = n-1$. Then equality holds in the Miyaoka-Yau inequality if and only if the pluricanonical map, $f: X \to X_{can}$, is a holomorphic fibre bundle away from singular fibres.
\end{corollary}

\noindent For arbitrary Kodaira dimension, we have the following extension.

\begin{theorem}
Let $X$ be a compact K\"ahler manifold with semi-ample canonical line bundle, $K_X$, smooth canonical model, $X_{can}$ and Kodaira dimension, $0 < \kappa \leq n-2$. If $[\omega_0] = 2 \pi c_1(L_0)$ for some ample $\Q$-line bundle $L_0$, then
\begin{equation*}
    (2(n+1) c_2(X) - n c_1(X)^2) \cdot c_1(L_0)^{n-\kappa-2} \cdot (-2 \pi c_1(X))^\kappa = 24(n+1) (n-\kappa-2)! r_{n-\kappa-2} [\eta_{FS}]^\kappa
\end{equation*}
where $r_{n-\kappa-2}$ is the coefficient of $j^{n-\kappa-2}$ in the polynomial $\rank(f_*(K_{X/X_{can}} \otimes L_0^j))$ for sufficiently large $j \geq 1$.
\end{theorem}

\noindent \textbf{Acknowledgements:} The author would like to thank Tiernan Cartwright, Minghao Miao, Yanir Rubinstein, Yuguang Zhang, Xiaohua Zhu, and his supervisor, Zhou Zhang, for their interest and comments.

\section{Miyaoka-Yau Inequalities}

\subsection{Semi-Ample Canonical Line Bundle}

\noindent In this subsection we present a proof of the following theorem via the K\"ahler-Ricci flow.

\begin{theorem}\label{semi_ample_MY}
Let $X$ be a compact K\"ahler manifold with semi-ample $K_X$. Then
\begin{equation*}
    (2 (n+1) c_2(X) - n c_1(X)^2) \cdot [\omega_0]^{n-i-2} \cdot (-2\pi c_1(X))^{i} \geq 0
\end{equation*}
for all $0 < [\omega_0] \in H^{1,1}(X,\R)$, where $n = \dim X$ and $\kappa = \Kod(X)$ satisfies $0 < \kappa < n$, and $i = \min\{\kappa,n-2\}$.
\end{theorem}

\noindent When $\kappa = n$, this result was proven by Zhang in \cite{Z09}. If $\kappa = 0$ then $X$ is Calabi-Yau and handled by Yau, \cite{Y77,K14}.\\

\noindent Taking $K_X$ to be semi-ample, and any initial K\"ahler metric $\omega_0$, the K\"ahler-Ricci flow
\begin{equation*}
    \frac{\d \omega}{\d t} = -\Ric \omega(t) - \omega(t), \quad \omega(0) = \omega_0
\end{equation*}
exists for all time, \cite{TZ06}. Moreover, a sufficiently large tensor power of the canonical line bundle, $K_X$, generates a proper, holomorphic map with connected fibres
\begin{equation*}
    f: X \to X_{can} \subseteq \P H^0(X,K_X^{\otimes k})
\end{equation*}
into projective space whose image, $X_{can}$, the canonical model, is an irreducible projective variety. We assume $0 < \kappa < n$ where $n = \dim X$, $\kappa = \Kod(X) = \dim X_{can}$. Next, note that since $f^*\OO(1) = K_X^{\otimes k}$ then $[f^*\eta_{FS}] = -2 \pi c_1(X)$ where we define $\eta_{FS} = \frac{1}{k} \omega_{FS}|_{X_{can}}$ and we can set $k$ to be the smallest sufficiently large positive integer which generates the map.\\

\noindent Note that the cohomology evolves along the K\"ahler-Ricci flow according to
\begin{equation*}
    [\omega(t)] = e^{-t} [\omega_0] + (1-e^{-t}) [f^*\eta_{FS}]
\end{equation*}
so that $C^{-1} e^{-(n -\kappa)t} \leq \Vol(X,\omega(t)) \leq C e^{-(n-\kappa)t}$.

\begin{proof}
Consider the case $0 < \kappa < n-2$. Firstly, note that by Chern-Weil theory, \cite{B87,K14},
\begin{align*}
    (2 (n+1) c_2(X) &- n c_1(X)^2) \cdot [\omega(t)]^{n-2}\\
    &= \frac{1}{4 \pi^2 n(n-1)}\int_X (n+1)|\Rm^\circ\omega(t)|^2 - (n+2)|\Ric^\circ\omega(t)|^2 \omega(t)^n
\end{align*}
where $\Rm^\circ$, $\Ric^\circ$, are the traceless Riemannian and Ricci curvatures given by
\begin{equation*}
    R^\circ_{i\oo{j}k\oo{l}} = R_{i\oo{j}k\oo{l}} - \frac{R}{n(n+1)}(g_{i\oo{j}}g_{k\oo{l}} + g_{i\oo{l}}g_{j\oo{k}}), \quad R^\circ_{i\oo{j}} = R_{i\oo{j}} - \frac{R}{n} g_{i\oo{j}}.
\end{equation*}
Therefore,
\begin{align*}
    (2 (n+1) c_2(X) &- n c_1(X)^2) \cdot e^{(n-\kappa-2)t}[\omega(t)]^{n-2}\\ &\geq -C(n)\int_X e^{(n-\kappa-2)t} |\Ric^\circ \omega(t)|^2 \omega(t)^n\\
    &\geq -C(n) \int_X e^{(n-\kappa-2)t}|\Ric \omega(t) + \omega(t)|^2 \omega(t)^n
\end{align*}
by applying
\begin{equation*}
    |\Ric \omega + \omega|^2  - |\Ric^\circ \omega|^2 = \frac{1}{n}(R + n)^2 \geq 0.  
\end{equation*}
Next, recall, \cite{Z09}, that the scalar curvature evolves along the flow according to
\begin{equation*}
    \parabolic{\omega(t)} R = R + |\Ric \omega|^2 = -(R+n) + |\Ric \omega + \omega|^2
\end{equation*}
and the volume form by
\begin{equation*}
    \d_t \omega^n = -(R+n)\omega^n.
\end{equation*}
Then we have
\begin{align*}
    \int_X e^{(n-\kappa-2)t}|\Ric \omega + \omega|^2 \omega^n
    &= e^{(n-\kappa-2)t} \int_X  (\d_t R +(R+n))\omega^n\\
    &= \frac{d}{dt} \int_X e^{(n-\kappa-2)t} R \omega^n - \int_X (n-\kappa-2) e^{(n-\kappa-2)t}R \omega^n\\
    &+ \int_X e^{(n-\kappa-2)t}R(R+n)\omega^n + \int_X e^{(n-\kappa-2)t} (R+n) \omega^n\\
    &\leq \frac{d}{dt} \int_X e^{(n-\kappa-2)t} R \omega^n  + Ce^{-2t}
\end{align*}
where the last line uses the uniform bound $\norm{R}_{C^0} \leq C$ for the K\"ahler-Ricci flow with semi-ample $K_X$, \cite{ST16}, in combination with the fact that $C^{-1} e^{-(n-\kappa)t} \leq [\omega(t)]^n \leq C e^{-(n-\kappa)t}$. Therefore
\begin{align*}
    \int_0^\infty \int_X e^{(n-\kappa-2)t} |\Ric \omega + \omega|^2 \omega^n dt \leq C,
\end{align*}
and so there exists a sequence $(t_i)_{i \in \N}$ such that $t_i \to \infty$ as $i \to \infty$ and 
\begin{equation*}
    \int_X e^{(n-\kappa-2)t_i} |\Ric \omega(t_i) + \omega(t_i)|_{\omega(t_i)}^2 \omega(t_i)^n \to 0
\end{equation*}
as $i \to \infty$. Finally, since $[\omega(t)] = e^{-t}[\omega_0] + (1-e^{-t}) [f^*\eta_{FS}]$ and $0 < \kappa < n - 2$, we have
\begin{align*}
    e^{(n-\kappa-2)t} [\omega(t)]^{n-2} &= e^{(n-\kappa-2)t} \sum_{i=0}^{n-2} {n-2 \choose i} e^{-(n-2-i)t} (1-e^{-t})^{i} [\omega_0]^{n-2-i} \cdot [f^*\eta_{FS}]^i\\
    &= \sum_{i=0}^{\kappa} {n-2 \choose i} e^{-(\kappa-i)t} (1-e^{-t})^{i} [\omega_0]^{n-2-i} \cdot [f^*\eta_{FS}]^i\\
    &\to {n-2 \choose \kappa} [\omega_0]^{n-\kappa-2} \cdot [f^*\eta_{FS}]^{\kappa}
\end{align*}
as $t \to \infty$. It is immediate that
\begin{align*}
     (2 (n+1) c_2(X) &- n c_1(X)^2) \cdot [\omega_0]^{n-\kappa-2} \cdot [f^*\eta_{FS}]^{\kappa}\\
    &= {n-2 \choose \kappa}^{-1}\lim_{i \to \infty} (2 (n+1) c_2(X) - n c_1(X)^2) \cdot e^{(n-\kappa-2)t_i}[\omega(t_i)]^{n-2}\\
    &\geq - C(\kappa,n) \lim_{i \to \infty} \int_X e^{(n-\kappa-2)t_i} |\Ric \omega(t_i) + \omega(t_i)|_{\omega(t_i)}^2 \omega(t_i)^n\\
    &= 0.
\end{align*}
Finally, consider the cases $\kappa = n-2, n-1$. It is clear that
\begin{equation*}
    \left(2(n+1) c_2(X) - n c_1(X)^2 \right)\cdot [\omega(t)]^{n-2} \geq -C(n) \int_X |\Ric \omega(t) + \omega(t)|^2 \omega(t)^n
\end{equation*}
as before, and again by the uniform scalar curvature bound,
\begin{align*}
    \int_X |\Ric \omega + \omega|^2 \omega^n &= \frac{d}{dt} \left(\int_X R \omega^n\right) + \int_X (R+1)(R+n) \omega^n\\
    &\leq \frac{d}{dt} \left(\int_X R \omega^n\right) + C e^{-t}.
\end{align*}
Therefore 
\begin{equation*}
    \int_0^\infty \int_X |\Ric \omega + \omega|^2 \omega^n dt \leq C
\end{equation*}
and the result follows as before since $[\omega(t)]^{n-2} \to [f^*\eta_{FS}]^{n-2}$ as $t \to \infty$.
\end{proof}

\noindent In fact, it is clear from the proof that
\begin{equation}\label{semi_ample_L2_ricci_limit}
    \int_0^\infty \int_X e^{(n-\kappa-1)t} |\Ric \omega + \omega|^2 \omega^n dt \leq C,
\end{equation}
but this does not enable a more optimal Miyaoka-Yau inequality than the stated one since the class $e^{(n-\kappa-1)t}[\omega(t)]^{n-2}$ blows up as $t \to \infty$.

\begin{remark}
We note that \cite{N21} claims to prove the Miyaoka-Yau inequality using the K\"ahler-Ricci flow for $K_X$ semi-positive which would include our setting, however there are two issues. Firstly, the estimate for the third term in \cite[Eq. (2.4)]{N21} contains an error and, secondly, is not optimal since the left-hand side of the inequality
\begin{equation*}
    (2 (n+1) c_2(X) - n c_1(X)^2) \cdot (-2 \pi c_1(X))^{n-2} \geq 0
\end{equation*}
is trivially $0$ whenever $0 < \kappa < n-2$.
\end{remark}

\noindent In fact, it has been proven by Liu, \cite{L23}, that when $K_X$ is merely nef, that Theorem \ref{semi_ample_MY} holds, where $\kappa$ is replaced with the numerical dimension of $-2\pi c_1(X)$. Recall that the numerical dimension of a nef class $[\alpha] \in H^{1,1}(X,\R)$ is defined to be 
\begin{equation*}
    \nu([\alpha]) = \max\{k \in \{0,\dots,n\}: [\omega_0]^{n-k} \cdot [\alpha]^k > 0\}
\end{equation*}
for any K\"ahler class $[\omega_0]$. Liu's technique relies on using cscK metrics to approximate the curvature terms appearing from Chern-Weil theory. The only reason we require semi-ampleness is that a uniform scalar curvature bound is currently unknown for the K\"ahler-Ricci flow when $K_X$ is only nef. There is no other dependence on the semi-ampleness of the canonical line bundle.\\

\noindent In section $3$ we prove the slope semi-stability of the holomorphic tangent bundle, $T^{1,0} X$, when $K_X$ is nef and we use the same approximation method with the cscK metrics. So take note of the following theorem of Dyrefelt and Song.

\begin{theorem}[\cite{D20},\cite{S20}]
Let $X$ be any compact K\"ahler manifold where $K_X$ is nef. Then for any K\"ahler class $[\omega_0]$, there exists an $\e_0 > 0$ such that $\e[\omega_0] - 2 \pi c_1(X)$ admits a unique cscK metric for all $0 < \e < \e_0$.
\end{theorem}

\noindent Denote the unique cscK metric in $\e [\omega_0] - 2 \pi c_1(X)$ by $\omega_{\e}$. For later use, we note that one can easily adapt the arguments of Liu, \cite{L23}, to prove
\begin{equation}\label{nef_L2_ricci_limit}
    \lim_{\e \to 0} \int_X \e^{-(n-\nu-1)}|\Ric \omega_{\e} + \omega_{\e}|^2 \omega_{\e}^n = 0.
\end{equation}

\subsection{Fano Manifolds}

The K\"ahler Ricci flow can also be used to prove the following Miyaoka-Yau inequality.

\begin{theorem}\label{fano_MY}
Let $X$ be any K-semistable Fano manifold. Then
\begin{equation*}
    (2(n+1) c_2(X) - n c_1(X)^2) \cdot (2\pi c_1(X))^{n-2} \geq 0.
\end{equation*}
\end{theorem}

\noindent Let $X$ be Fano and consider the normalized K\"ahler-Ricci flow
\begin{equation*}
    \frac{\d \omega(t)}{\d t} = -\Ric 
    \omega(t) + \omega(t), \quad \omega(0) = \omega_0 \in 2 \pi c_1(X),
\end{equation*}
which exists for all time, \cite{TZ06,BEG13}, and satisfies $[\omega(t)] = 2 \pi c_1(X)$. 

\begin{proposition}\label{L2_ricci_scalar_equality}
We have
 \begin{equation*}
     \int_X |\Ric \omega - \omega|_{\omega}^2 \omega^n = \int_X (R-n)^2 \omega^n. 
 \end{equation*}
\end{proposition}

\begin{proof}
We have, for some $A \in \R$ to be determined, that the scalar curvature evolves according to, \cite{TZ16},
\begin{equation*}
    \parabolic{\omega} (R + A) = (R - n) + |\Ric \omega - \omega|_{\omega}^2
\end{equation*}
and the volume form by
\begin{equation*}
    \d_t \omega^n = -(R-n)\omega^n.
\end{equation*}
Then it follows that
\begin{align*}
    \int_X |\Ric \omega - \omega|^2_{\omega} \omega^n &= \int_X \d_t(R +A) \omega^n - \int_X (R-n )\omega^n\\
    &= \frac{d}{dt} \int_X (R + A) \omega^n - \int_X (R + A) \d_t\omega^n - \int_X (R- n) \omega^n\\
    &= \frac{d}{dt} \int_X (R+A) \omega^n + \int_X (R-n)^2 \omega^n\\
    &= \int_X (R-n)^2 \omega^n
\end{align*}
by setting $A = 1-n$ since $\int_X (R +A) \omega^n$ is constant. Indeed, since $[\omega(t)] = [\omega_0] = 2 \pi c_1(X)$, we find
\begin{equation*}
    \int_X (R_{\omega(t)} + A) \omega(t)^n = \int_X \Big(n \Ric \omega(t) \wedge \omega(t)^{n-1} + A \omega(t)^n\Big) =  (n + A) \int_X (2 \pi c_1(X))^n.
\end{equation*}
\end{proof}

\noindent Now observe the following theorem. We note that a lower bound on the Mabuchi K-energy on $2 \pi c_1(X)$ is equivalent to K-semistability, \cite{L17}, and does not guarantee the existence of a K\"ahler-Einstein metric.

\begin{theorem}[\cite{PSSW09}]
If the Mabuchi K-energy is bounded from below on $2 \pi c_1(X)$ then the normalized K\"ahler-Ricci flow satisfies
\begin{equation*}
    \norm{R_{\omega(t)} - n}_{C^0} \to 0
\end{equation*}
as $t \to \infty$.
\end{theorem}

\noindent In general, it is known due to Perelman that the scalar curvature is uniformly bounded along the flow, \cite{SesT08}. Now on to the Miyaoka-Yau inequality. 

\begin{proof}
It follows from $ |\Ric\omega - \omega|^2_{\omega} - |\Ric^\circ \omega|^2_{\omega} = \frac{1}{n}(R - n)^2 \geq 0$ and Proposition \ref{L2_ricci_scalar_equality} that
\begin{align*}
    \big(2(n+1)c_2(X) &- n c_1(X)^2\big) \cdot [\omega(t)]^{n-2}\\
    &= \frac{1}{4 \pi^2 n (n-1)}\int_X (n+1)|\Rm^\circ \omega(t)|^2_{\omega(t)} - (n+2)|\Ric^\circ \omega(t)|^2_{\omega(t)} \omega(t)^n\\
    &\geq -C(n) \int_X |\Ric\omega - \omega|^2_{\omega} \omega^n\\
    &= -C(n) \int_X (R-n)^2 \omega^n.
\end{align*}
\noindent The desired inequality is immediate, i.e.
\begin{align*}
    (2(n+1)c_2(X) - n c_1(X)^2) \cdot (2 \pi c_1(X))^{n-2} \geq -C(n) \norm{R - n}_{C^0}^2 \int_X (2 \pi c_1(X))^n  \to 0
\end{align*}
as $t \to \infty$. 
\end{proof}

\noindent Moreover, we can extend this result to the following setting. See \cite{TZ02,PSSW11} for the definition of the $V$-modified Mabuchi K-energy.

\begin{theorem}\label{modified_fano_MY}
Let $X$ be any Fano manifold, where $V$ is a non-trivial holomorphic vector field whose imaginary
part, $\Imag V$, induces an $S^1$-action on $X$. If the $V$-modified Mabuchi K-energy is bounded below on $2 \pi c_1(X)$ then, for some constant $C_V > 0$ defined by (\ref{divergence_norm}), we have
\begin{equation*}
    (2(n+1) c_2(X) - n c_1(X)^2) \cdot  c_1(X)^{n-2} \geq -\frac{C_V(n+2)}{n(n-1)} c_1(X)^n.
\end{equation*}
\end{theorem}

\noindent The proof requires the $V$-modified K\"ahler-Ricci flow. We say $(X,\omega(t))$ satisfies the $V$-modified K\"ahler-Ricci flow if
\begin{equation*}
    \frac{\d\omega}{\d t} = -\Ric \omega(t) + \omega(t) + L_{V} \omega(t), \quad \omega(0) = \omega_0 \in 2 \pi c_1(X).
\end{equation*}
Again, the flow exists for all time and satisfies $[\omega(t)] \in 2 \pi c_1(X)$, \cite{PSSW11}.\\

\begin{theorem}[\cite{PSSW11}]\label{scalar_and_div_bounds}
There exists a $C > 0$ such that the $V$-modified K\"ahler-Ricci flow satisfies
\begin{equation*}
    \norm{R_{\omega(t)}}_{C^0} + \norm{\div_{\omega(t)} V}_{C^0} \leq C.
\end{equation*}
If the $V$-modified Mabuchi K-energy is bounded from below on $2\pi c_1(X)$ then the $V$-modified K\"ahler-Ricci flow satisfies
\begin{equation*}
    \norm{R_{\omega(t)} - n - \div_{\omega(t)} V}_{C^0} \to 0
\end{equation*}
as $t \to \infty$.
\end{theorem}

\noindent In particular, this means that 
\begin{equation}\label{divergence_norm}
    C_V := \limsup_{t \to \infty} \norm{\div_{\omega(t)} V}_{C^0}^2
\end{equation}
is some finite, non-negative constant. We denote $\theta_{V}(t)$ to be the smooth function such that $\iota_V \omega(t) = \dbar \theta_V(t)$, and normalize by
\begin{equation*}
    \int_X (e^{\theta_V(t)} -1)\omega(t)^n =0.
\end{equation*}
It is not difficult to check
\begin{equation*}
    \div_{\omega(t)} V = \Delta_{\omega(t)} \theta_V.
\end{equation*}

\begin{proposition}
The scalar curvature evolves according to
\begin{equation*}
    \parabolic{} R = (R-n) + |\Ric \omega - \omega|^2 - \ip{i\d\dbar \theta_V, \Ric \omega} - \Delta \Delta \theta_V,
\end{equation*}
and the volume form evolves according to
\begin{equation*}
    \d_t \omega^n = -(R - n - \Delta \theta_V ) \omega^n.
\end{equation*}
\end{proposition}

\begin{proof}
One can check
\begin{align*}
    \d_t \tr_{\omega} \alpha &= \ip{\Ric \omega, \alpha} - \tr_{\omega} \alpha - \ip{i\d\dbar \theta_V, \alpha} + \tr_{\omega} (\d_t \alpha)\\
    \d_t \Ric \omega &= i\d\dbar (R - n - \Delta \theta_V)
\end{align*}
and the evolution of the scalar curvature will follow. The evolution of the volume form follows by direction computation.
\end{proof}

\noindent Now let's prove the Miyaoka-Yau inequality.

\begin{proof}
Firstly,
\begin{equation*}
    (2(n+1) c_2(X) - n c_1(X)^2) \cdot [\omega(t)]^{n-2} \geq - \frac{n+2}{4 \pi^2 n(n-1)} \int_X |\Ric^\circ \omega|^2 \omega^n.
\end{equation*}
Then, by the evolution of the scalar curvature and the volume form,
\begin{align*}
    \int_X |\Ric^\circ \omega|^2 \omega^n &= \int_X \left(- \frac{1}{n}(R-n)^2 + |\Ric \omega - \omega|^2 \right)\omega^n \\
    &= \int_X \left(- \frac{1}{n}(R-n)^2 + \d_t R - (R-n) + \ip{i\d\dbar \theta_V, \Ric \omega} \right)\omega^n \\
    &= \frac{d}{dt} \left(\int_X R \omega^n\right) + \int_{X} \ip{i\d\dbar \theta_V, \Ric^\circ \omega} \omega^n\\
    &+\int_X \left(\frac{-R}{n}\right)(R - n - \Delta \theta_V) \omega^n
\end{align*}
Note that
\begin{equation*}
    \frac{d}{dt} \left(\int_X R \omega^n\right) = \frac{d}{dt} \int_X n (2\pi c_1(X))^n = 0
\end{equation*}
and, moreover, for some $\e > 0$ to be determined,
\begin{align*}
    \int_X \ip{i\d\dbar \theta_V, \Ric^{\circ} \omega} \omega^n &\leq \int_X \left(\frac{1}{2 \e} |\nabla \oo{\nabla} \theta_V|^2 + \frac{\e}{2} |\Ric^{\circ} \omega |^2 \right)\omega^n\\
    &= \int_X \frac{1}{2 \e} (\Delta \theta_V)^2 + \frac{\e}{2} |\Ric^{\circ} \omega|^2 \omega^n,
\end{align*}
so that
\begin{align*}
    \int_X |\Ric^\circ \omega|^2 \omega^n &\leq \frac{2}{2-\e}\int_X \left(-\frac{R}{n}\right) (R - n - \Delta \theta_V) \omega^n \\
    &+ \frac{1}{2\e - \e^2} \int_X (\Delta \theta_V)^2 \omega^n.
\end{align*}
Choose $\e = 1$ such that $\frac{1}{2\e - \e^2}$ is minimized. Thus, by applying Theorem \ref{scalar_and_div_bounds}, we arrive at
\begin{align*}
    (2(n+1) c_2(X) - n c_1(X)^2) \cdot c_1(X)^{n-2} &\geq -\frac{n+2}{(2\pi)^{n}n(n-1)} \limsup_{t \to \infty} \int_X |\Ric^\circ \omega|^2 \omega^n\\
    &\geq -\frac{n+2}{(2\pi)^n n(n-1)} \limsup_{t \to \infty}\norm{\Delta \theta_V}^2_{C^0} \int_X (2 \pi c_1(X))^n\\
    &- C \lim_{t \to \infty} \norm{R - n - \Delta \theta_V}_{C^0} \int_X (2 \pi c_1(X))^n
\end{align*}
which gives the desired result.
\end{proof}

\noindent One could also consider using the twisted K\"ahler-Ricci flow, \cite{L13}, to derive some Miyaoka-Yau type inequality.

\section{Slope Semi-Stability}

\noindent We recall the notion of slope $[\omega_0]$-(semi)stability, \cite{UY86,K14}. The slope of a holomorphic vector bundle, $E$, with respect to a K\"ahler class, $[\omega_0]$, is the quantity
\begin{equation*}
    \mu(E) = \frac{\int_X 2\pi c_1(E) \wedge [\omega_0]^{n-1}}{ \rank (E)} = \frac{\deg(E)}{\rank(E)}.
\end{equation*}
A holomorphic vector bundle $E$ is slope $[\omega_0]$-(semi)stable if for every coherent, proper $\OO_X$-subsheaf $\mathcal{F} \subset E$, i.e. $0 < \rank(\mathcal{F}) < \rank(E)$, 
\begin{equation*}
    \mu(\mathcal{F}) < \mu(E) \quad (\mu(\mathcal{F}) \leq \mu(E)),
\end{equation*}
respectively. We say a holomorphic vector bundle is slope $L$-(semi)stable, for some holomorphic line bundle $L$, if it is $2\pi c_1(L)$-(semi)stable.

\begin{remark}
The first Chern class maybe be extended to such coherent sheaves via $c_1(\mathcal{F}) := c_1((\det \mathcal{F})^{**})$, since $(\det \mathcal{F})^{**}$ is a holomorphic line bundle.
\end{remark}

\noindent We add the following refinement of slope and stability for nef classes.

\begin{definition}\label{normalized_slope}
Let $E$ be a holomorphic vector bundle, $[\eta]$ a nef class, and $[\omega_0]$ a K\"ahler class. We define the $([\omega_0], [\eta])$-slope of $E$ to be
\begin{equation*}
    \mu_{[\omega_0],[\eta]}(E) = \frac{\int_X 2 \pi c_1(E) \wedge \omega_0^{n-i-1} \wedge \eta^{i}}{\rank(E) \int_X \omega_0^{n-i} \wedge \eta^{i}}
\end{equation*}
where $\nu = \nu([\eta])$ and $i = \min\{\nu, n-1\}$. We say $E$ is $([\omega_0], [\eta])$-(semi)stable if
\begin{equation*}
    \mu_{[\omega_0],[\eta]}(\mathcal{F}) < \mu_{[\omega_0],[\eta]}(E), \quad (\mu_{[\omega_0],[\eta]}(\mathcal{F}) \leq \mu_{[\omega_0],[\eta]}(E))
\end{equation*}
for every coherent, proper $\OO_X$-subsheaf $\mathcal{F}\subset E$, respectively.
\end{definition}

\noindent For convenience, we simply write $([\omega_0],L)$-slope or semistable instead of $([\omega_0], 2 \pi c_1(L))$-slope or semistable, etc. Moreover, note that if $[\alpha]$ is a K\"ahler class then 
\begin{equation*}
    \mu_{[\omega_0], [\alpha]}(E) = \frac{\int_X 2 \pi c_1(E) \wedge [\alpha]^{n-1}}{\rank(E) \int_X [\omega_0] \wedge [\alpha]^{n-1}}
\end{equation*}
so that $E$ is $([\omega_0],[\alpha])$-(semi)stable iff it is $[\alpha]$-(semi)stable in the classical sense. We remark that, for the classic slope, one may have $\mu(E) = 0$ if $[\omega_0]$ is replaced with an arbitrary nef class, and so the notions of slope and stability become trivial. This normalized slope, $\mu_{[\omega_0],[\eta]}(E)$, avoids this issue and has the following nice property.

\begin{proposition}
Suppose $K_X$ is nef, so that $[\omega_{\e}] = \e[\omega_0] - 2 \pi c_1(X)$ is K\"ahler for all $0 < \e < \e_0$ for some $\e_0 > 0$. Then
\begin{equation*}
    \mu_{[\omega_0], K_X}(E) = \lim_{\e \to 0} \mu_{[\omega_0], [\omega_\e]}(E).
\end{equation*}
\end{proposition}

\begin{proof}
By the binomial theorem, if $\nu \leq n-1$,
\begin{equation*}
    \lim_{\e \to 0} \e^{-(n-\nu-1)} [\omega_{\e}]^{n-1} = {n-1 \choose \nu} [\omega_0]^{n -\nu - 1} \cdot (-2 \pi c_1(X))^\nu
\end{equation*}
so that
\begin{equation*}
    \lim_{\e \to 0} \frac{\int_X 2 \pi c_1(E) \wedge [\omega_{\e}]^{n-1}}{\rank(E) \int_X [\omega_0] \wedge [\omega_\e]^{n-1}} = \frac{\int_X 2 \pi c_1(E) \wedge [\omega_0]^{n-\nu-1} \wedge (-2 \pi c_1(X))^\nu}{\rank(E)\int_X [\omega_0]^{n-\nu} \wedge (-2 \pi c_1(X))^{\nu}} 
\end{equation*}
as desired. It is trivial for $\nu = n$.
\end{proof}

\noindent In regards to proving the normalized slope semi-stability of the holomorphic tangent bundle in certain cases, we have the following useful lemma (c.f. \cite[Prop. 5.8.2.]{K14}).

\begin{lemma}
Let $X$ be any compact K\"ahler manifold. For any coherent, proper subsheaf $\mathcal{F}$ of $T^{1,0} X$, and K\"ahler classes $[\omega_0]$, $[\omega_1]$,
\begin{equation*}
    \mu_{[\omega_0], [\omega_1]}(\mathcal{F}) - \mu_{[\omega_0], [\omega_1]}(T^{1,0} X) \leq  \frac{\int_X |\Ric \omega_1 - \lambda_{\omega_1} \omega_1|_{\omega_1} \omega_1^n}{n\int_X \omega_0 \wedge \omega_1^{n-1}}
\end{equation*}
where $\lambda_{\omega_1} = \frac{\deg_{\omega_1} (T^{1,0} X)}{\Vol_{\omega_1}(X)}$.
\end{lemma}

\begin{proof}
Let $F \subset T^{1,0} X$ be a holomorphic vector bundle and denote $\rank(F) = r$, and let $h_1$ be the Hermitian metric on $T^{1,0} X$ induced by $\omega_1$. Naturally, $h_1$ induces Hermitian metrics on $F$ and $\det F$. Let $\pi_{h_1}: T^{1,0}X \to F$ denote orthogonal projection with respect to $h_1$, so that the Gauss-Codazzi equation, \cite{GH94}, yields
\begin{equation*}
    \Theta_{\det F,h_1} = \tr(\Theta_{F,h_1})= \tr (\pi_{h_1} \Theta_{T^{1,0} X, h_1}|_F) - \tr(\oo{\II}^T \wedge \II)
\end{equation*}
where $\II$ is the second fundamental form of $F \hookrightarrow T^{1,0} X$ and the trace acts on the endomorphism component of the curvature forms $\Theta$. Therefore, applying the dual-Lefschetz operator, $\Lambda_{\omega_1}$, yields
\begin{equation*}
    \Lambda_{\omega_1} \Theta_{\det F, h_1} \leq \tr(\pi_{h_1} \Lambda_{\omega_1}\Theta_{T^{1,0}X, h_1}|_F).
\end{equation*}
It follows that
\begin{align*}
    \mu(F) - \mu(T^{1,0}X) &= \frac{1}{r} \int_X 2 \pi c_1(F) \wedge [\omega_1]^{n-1} - \mu(T^{1,0}X)\\
    &= \frac{1}{rn} \int_X (\Lambda_{\omega_1} \Theta_{\det F, h_1}) \omega_1^n - \mu(T^{1,0}X)\\
    &\leq \frac{1}{rn}\int_X \tr(\pi \Lambda \Theta_{T^{1,0}X}) \omega_1^n - \mu(T^{1,0} X)\\
    &= \frac{1}{rn} \int_X \tr(\pi E) \omega_1^n
\end{align*}
by setting $E = \Lambda_{\omega_1} \Theta_{T^{1,0}X} - \lambda_{\omega_1} \id_{T^{1,0}X}$ and $\lambda_{\omega_1} = \frac{\deg_{\omega_1}(T^{1,0}X)}{\Vol_{\omega_1}(X)}$. Thus
\begin{align*}
    \mu_{[\omega_0], [\omega_1]}(F) - \mu_{[\omega_0],[\omega_1]}(T^{1,0}X) &\leq \frac{ \int_X \tr(\pi E) \omega_1^n}{rn \int_X \omega_0 \wedge \omega_1^{n-1}}\\
    &\leq \frac{ \int_X |E|\omega_1^n}{n \int_X \omega_0 \wedge \omega_1^{n-1}}\\
    &= \frac{\int_X |\Ric \omega_1 - \lambda_{\omega_1} \omega_1|_{\omega_1} \omega_1^n}{n \int_X \omega_0 \wedge \omega_1^{n-1}}
\end{align*}
by noting $E= (\Ric \omega_1 - \lambda_{\omega_1} \omega_1)^{\sharp}$.\\

\noindent Next, assume that $\mathcal{F}$ is a coherent, proper subsheaf of $E$. Then $\mathcal{F}$ is torsion-free and, by \cite[Cor. 5.5.15.]{K14}, there exists an analytic Zariski-open set $U \subseteq X$ with $\codim_{\C}(X\backslash U) \geq 2$ such that $\mathcal{F}|_U$ is locally free. Hence we observe
\begin{align*}
    \deg(\mathcal{F}) &= \int_X 2 \pi c_1(\mathcal{F}) \wedge \omega_1^{n-1} = \int_U 2 \pi c_1(\mathcal{F})|_U \wedge \omega_1^{n-1}\\ 
    &= \int_U 2 \pi c_1((\det \mathcal{F})^{**}|_U) \wedge \omega_1^{n-1} = \int_U 2 \pi c_1(\det(\mathcal{F}|_U)) \wedge \omega_1^{n-1} = \deg(\mathcal{F}|_U).
\end{align*}
Since $\mathcal{F}|_U$ is a holomorphic vector bundle we can apply the previous calculation to find
\begin{align*}
    \mu_{[\omega_0],[\omega_1]}(\mathcal{F}) - \mu_{[\omega_0],[\omega_1]}(T^{1,0} X) &= \mu_{[\omega_0],[\omega_1]}(\mathcal{F}|_U) - \mu_{[\omega_0],[\omega_1]}(T^{1,0}U)\\
    &\leq \frac{\int_U |\Ric \omega_1 - \lambda_{\omega_1,U} \omega_1|_{\omega_1} \omega_1^n}{n\int_U \omega_0 \wedge \omega_1^{n-1}}\\
    &\leq \frac{\int_X |\Ric \omega_1 - \lambda_{\omega_1,X} \omega_1|_{\omega_1} \omega_1^n }{n \int_X \omega_0 \wedge \omega_1^{n-1}},
\end{align*}
since $\lambda_{\omega_1,U} := \frac{\deg_{\omega_1}(T^{1,0}U)}{\Vol_{\omega_1}(U)} = \lambda_{\omega_1,X}$.
\end{proof}

\noindent This leads to the following stability results.

\begin{theorem}
For any compact K\"ahler manifold $X$ with nef $K_X$, the holomorphic tangent bundle, $T^{1,0} X$, is slope $([\omega_0], K_X)$-semistable for all K\"ahler classes $[\omega_0]$.
\end{theorem}

\begin{proof}
As before, let $\omega_\e$ denote the unique cscK metric in $\e [\omega_0] - 2\pi c_1(X)$. We have
\begin{align*}
    \mu_{[\omega_0], K_X}(\mathcal{F}) &- \mu_{[\omega_0], K_X}(T^{1,0} X)\\ &= \lim_{\e \to 0} \left(\mu_{[\omega_0], [\omega_\e]}(\mathcal{F}) - \mu_{[\omega_0], [\omega_\e]}(T^{1,0} X) \right)\\
    &\leq \lim_{\e \to 0} \frac{\int_X |\Ric \omega_{\e} - \lambda_{\e} \omega_{\e}| \omega_\e^n}{n \int_X \omega_0 \wedge \omega_{\e}^{n-1}}\\
    &\leq \lim_{\e \to 0 } \frac{\left(\int_X |\Ric \omega_{\e} + \omega_{\e}|^2 \omega_\e^n\right)^{1/2} \Vol_{\omega_{\e}}(X)^{1/2} + \sqrt{n}|1 + \lambda_{\e}| \Vol_{\omega_{\e}}(X)}{n \int_X \omega_0 \wedge \omega_{\e}^{n-1}}\\
    &=\lim_{\e \to 0} \frac{\left(\int_X \e^{-(n-\nu-1)}|\Ric \omega_{\e} + \omega_{\e}|^2 \omega_\e^n\right)^{1/2} (\e^{-(n-\nu-1)}\Vol_{\omega_{\e}}(X))^{1/2}}{n \int_X \omega_0 \wedge \e^{-(n-\nu-1)}\omega_{\e}^{n-1}}\\
    &+ \lim_{\e \to 0} \frac{\sqrt{n}|1 + \lambda_{\e}| \e^{-(n-\nu-1)}\Vol_{\omega_{\e}}(X)}{n \int_X \omega_0 \wedge \e^{-(n-\nu-1)}\omega_{\e}^{n-1}}\\
    &= 0,
\end{align*}
where we note that $\lambda_{\e} = \frac{1}{n}R_{\omega_\e}\to -\frac{\nu}{n}$ by, \cite{L23}, $\int_X \e^{-(n-\nu-1)}|\Ric \omega_{\e} + \omega_{\e}|^2 \omega_{\e}^n \to 0$ as $\e \to 0$ by our earlier remark, (\ref{nef_L2_ricci_limit}), and the fact that
\begin{equation*}
    [\omega_0] \cdot \e^{-(n-\nu-1)}[\omega_\e]^{n-1} \to {n-1 \choose \nu} [\omega_0]^{n-\nu} \cdot (-2 \pi c_1(X))^\nu, \quad \e^{-(n-\nu-1)}[\omega_\e]^n \to 0,
\end{equation*}
as $\e \to 0$.
\end{proof}

\noindent Note that when $K_X$ is semi-ample, the same argument applies using the K\"ahler-Ricci flow as well since $|\lambda_{\omega(t)}| \leq C$ by Song-Tian's scalar curvature bound, \cite{ST16}, and equation (\ref{semi_ample_L2_ricci_limit}).

\begin{theorem}
For any K-semistable Fano manifold $X$, the holomorphic tangent bundle, $T^{1,0} X$, is slope $-K_X$-semistable.
\end{theorem}

\noindent Let us denote the normalized slope as $\mu_{K_X^{-1}} = \mu_{[\omega_0], K_X^{-1}}$ when $2 \pi c_1(X) = [\omega_0]$.

\begin{proof}
Using the Fano K\"ahler-Ricci flow, we have 
\begin{align*}
    \mu_{K_X^{-1}}(\mathcal{F}) - \mu_{K_X^{-1}}(T^{1,0} X) &= \lim_{t \to \infty} \left(\mu_{\omega(t)}(\mathcal{F}) - \mu_{\omega(t)}(T^{1,0} X)\right)\\
    &\leq \lim_{t \to \infty}  \frac{\int_X |\Ric \omega(t) - \lambda(t) \omega(t)| \omega(t)^n}{n \int_X \omega(t)^n}\\
    &\leq \lim_{t \to \infty} \frac{\left(\int_X |\Ric \omega(t) - \omega(t)|^2 \omega(t)^n\right)^{1/2} \Vol_{\omega(t)}(X)^{1/2}}{\Vol_{\omega(t)}(X)}\\
    & = 0,
\end{align*}
by observing that $\lambda(t) = 1$ and $[\omega(t)] = 2\pi c_1(X)$ for all $t \in [0,\infty)$ and applying earlier estimates. Since $T^{1,0}X$ is $([\omega_0], K_X^{-1})$-semistable it is $K_X^{-1}$-semistable. 
\end{proof}

\begin{proposition}
Let $X$ be any Fano manifold, where $V$ is a non-trivial holomorphic vector field whose imaginary
part, $\Imag V$, induces an $S^1$-action on $X$, and the $V$-modified Mabuchi K-energy is bounded below on $2 \pi c_1(X)$. We have
\begin{equation*}
    \mu_{K_X^{-1}}(\mathcal{F}) - \mu_{K_X^{-1}}(T^{1,0} X) \leq \left(\frac{\sqrt{n}+1}{n\sqrt{n}}\right) C_V^{1/2},
\end{equation*}
for any coherent, proper subsheaf $\mathcal{F} \subset T^{1,0} X$.
\end{proposition}

\begin{proof}
Using the $V$-modified K\"ahler-Ricci flow, again this follows from the fact that $\lambda(t) = 1$ and $[\omega(t)] = 2 \pi c_1(X)$ for all $t \in [0,\infty)$, and 
\begin{align*}
    \mu_{K_X^{-1}}(\mathcal{F}) - \mu_{K_X^{-1}}(T^{1,0}X) &\leq
    \limsup_{t \to \infty}  \frac{\int_X |\Ric \omega(t) - \lambda(t) \omega(t)| \omega(t)^n}{n \int_X \omega(t)^n}\\
    &\leq \limsup_{t \to \infty} \frac{\left(\int_X |\Ric^\circ \omega|^2 \omega^n \right)^{1/2} \Vol_{\omega(t)}(X)^{1/2} + \frac{1}{\sqrt{n}}\int_X |R - n| \omega^n}{n\Vol_{\omega(t)}(X)}\\
    &\leq \frac{(C_V(2 \pi c_1(X))^n )^{1/2} ((2 \pi c_1(X))^n)^{1/2} + \frac{1}{\sqrt{n}} C_V^{1/2} (2 \pi c_1(X))^n}{n (2 \pi c_1(X))^n}\\
    &= \left(\frac{\sqrt{n}+1}{n\sqrt{n}}\right) C_V^{1/2}.
\end{align*}
\end{proof}

\section{Connection with the Weil-Petersson Metric}

\noindent We would also like some geometric characterizations of equality in the Miyaoka-Yau inequalities in Section $2$. In the case that $X$ is a smooth projective variety and $K_X$ is nef and big, Greb-Kebekus-Peternell-Taji, \cite{GKPT19}, prove that if equality holds in Theorem \ref{semi_ample_MY} then there exists a ball quotient $Y$ such that $X_{can}$ admits a finite, Galois, quasi-etale morphism $f: Y \to X_{can}$. In the case that $X$ is K-semistable Fano, then equality holds in Theorem \ref{fano_MY} if and only if it admits a finite, codimension-one, etale cover $f: \P^n \to X$, \cite{H24}. In the case that we have an elliptic fibration of a K\"ahler three-fold over a K\"ahler surface, it is known due to Wang, \cite{W26}, that equality holds in Theorem \ref{semi_ample_MY} if and only if the $j$-invariant of the fibers is constant. We extend this to the following setting.

\begin{theorem}\label{MY_equality}
If $X$ is any compact K\"ahler manifold with semi-ample $K_X$, smooth canonical model, $X_{can}$, and $\Kod(X) = n-1$, then $f: X \to X_{can}$ is a holomorphic fiber bundle away from singular fibres if and only if 
\begin{equation*}
    (2(n+1) c_2(X) - n c_1(X)^2) \cdot(- 2 \pi c_1(X))^{n-2} = 0.
\end{equation*}
\end{theorem}

\noindent This will be a consequence of the following result which establishes an intriguing connection between the Miyaoka-Yau quantity and the Weil-Petersson metric. First, note that due to Tian, \cite{T87},\cite[Eqn. (3.18)]{ST12}, the Weil-Petersson metric, $\omega_{WP}$, is the known to be the curvature form of a Hermitian metric on the $(n-\kappa)$-Hodge bundle $f_* \Omega^{n-\kappa}_{X/X_{can}}$ over the regular part of the variety $X_{can}$ and excluding the image of critical points of the map, and by \cite{CS21} extends to a $d$-closed positive $(1,1)$-current over $X_{can}$ when the canonical model is smooth, and represents $2 \pi c_1(f_* K_{X/X_{can}})$.

\begin{theorem}\label{MY_WP}
Let $X$ be any compact K\"ahler manifold with semi-ample $K_X$ and smooth canonical model, $X_{can}$. If $\Kod(X) = n-1$ then
\begin{equation*}
    (2(n+1) c_2(X) - n c_1(X)^2) \cdot (-2 \pi c_1(X))^{n-2} = \frac{12(n+1)}{\pi}[\omega_{WP}] \cdot [\eta_{FS}]^{n-2}.
\end{equation*}
\end{theorem}

\begin{proof}
Since $X_{can}$ is a smooth we may apply the relative Grothendieck-Riemann-Roch theorem for holomorphic maps between complex manifolds, \cite{BTT85}, to $\OO_X$, i.e.
\begin{equation*}
    \ch(Rf_* \OO_X) = f_* (\ch(\OO_X)\td(\T_{f})),
\end{equation*}
where $\T_f = \T_X - f^* \T_{X_{can}}$ denotes the virtual relative tangent sheaf in the Grothendieck group of coherent analytic sheaves, and $\T_X$, $\T_{X_{can}}$, the holomorphic tangent bundles. Firstly, since
\begin{equation*}
    \ch(\OO_X) = 1 + c_1(\OO_X) + \frac{1}{2!} c_1(\OO_X)^2 + \cdots + \frac{1}{n!} c_1(\OO_X)^n = 1
\end{equation*}
and 
\begin{equation*}
    \td(\T_f) = 1 + \frac{c_1(\T_f)}{2} + \frac{c_1(\T_f)^2 + c_2(\T_f)}{12} +  \frac{c_1(\T_f) \wedge c_2(\T_f)}{24} + \dots
\end{equation*}
and $\Kod(X) = n-1$, we find
\begin{equation*}
    \ch_1(R f_* \OO_X) = f_* \left(\frac{c_1(\T_f)^2 + c_2(\T_f)}{12}\right).
\end{equation*}
On the other hand, since the fibres are $1$-dimensional, the higher direct images of $\OO_X$ vanish, i.e. for $i > 1$, and since $f$ has connected fibres,
\begin{equation*}
    Rf_* \OO_X = \sum_{i \geq 0} (-1)^i R^i f_* \OO_X = f_* \OO_X - R^1 f_* \OO_X = \OO_{X_{can}} - (f_*K_{X/X_{can}})^*
\end{equation*}
where the last equality follows from relative Serre duality for derived functors. Therefore,
\begin{equation*}
    \ch_1(R f_* \OO_X) = c_1(R f_* \OO_X) = c_1(f_* K_{X/X_{can}}) = \frac{1}{2\pi}[\omega_{WP}],
\end{equation*}
and
\begin{equation*}
    \frac{1}{2\pi} [\omega_{WP}] = f_* \left(\frac{c_1(\T_f)^2 + c_2(\T_f)}{12}\right).
\end{equation*}
Next, since $\T_f = \T_X - f^*\T_{X_{can}}$, observe that the first and second Chern classes satisfy the following relations
\begin{align*}
    c_1(\T_X) &= c_1(\T_f) + c_1(f^* \T_{X_{can}}),\\
    c_2(\T_X) &= c_2(\T_f)+ c_1(\T_f) \wedge c_1 (f^* \T_{X_{can}}) + c_2(f^* \T_{X_{can}}),
\end{align*}
and we apply these in the following calculation,
\allowdisplaybreaks
\begin{align*}
    \int_{X_{can}} \frac{12(n+1)}{\pi} &[\omega_{WP}] \wedge [\eta_{FS}]^{n-2}\\
    =& \int_{X_{can}} 24(n+1) f_* \left(\frac{c_1(\T_f)^2 + c_2(\T_f)}{12}\right) \wedge [\eta_{FS}]^{n-2}\\
    =& \int_X 2(n+1) (c_1(\T_f)^2 + c_2(\T_f)) \wedge f^*\eta_{FS}^{n-2}\\
    =& \int_X 2(n+1)\left(c_1(X) - f^*c_1(X_{can})\right)^2 \wedge f^*\eta_{FS}^{n-2}\\
    +&\int_X 2(n+1) \Big( c_2(X) - c_1(\T_f) \wedge f^* c_1(X_{can}) - f^* c_2(X_{can})\Big) \wedge f^*\eta_{FS}^{n-2}\\ 
    =& \int_X 2(n+1)\left(\frac{1}{2\pi}f^*\eta_{FS} + f^*c_1(X_{can})\right)^2 \wedge f^*\eta_{FS}^{n-2}\\
    +&\int_X 2(n+1) \Big( c_2(X) + \Big(\frac{1}{2\pi} f^*\eta_{FS} + f^*c_1(X_{can})\Big) \wedge f^* c_1(X_{can})\\
    &- f^* c_2(X_{can})\Big) \wedge f^*\eta_{FS}^{n-2}\\ 
    =& \int_X 2(n+1) c_2(X) \wedge f^*\eta_{FS}^{n-2}\\
    =& \int_X (2(n+1) c_2(X) - n c_1(X)^2) \wedge (-2 \pi c_1(X))^{n-2}.
\end{align*}
Note that the pullback of an $(n,n)$-class must be $0$ and $-2\pi c_1(X) = [f^*\eta_{FS}]$.
\end{proof}

\noindent In fact, the semi-positivity of the Weil-Petersson metric yields a new proof of the Miyaoka-Yau inequality.

\begin{corollary}
Let $X$ be a compact K\"ahler manifold with semi-ample $K_X$, smooth canonical model, $X_{can}$, and $\Kod(X) = n-1$, then
\begin{equation*}
    (2(n+1) c_2(X) - n c_1(X)^2) \cdot (-2 \pi c_1(X))^{n-2} \geq 0.
\end{equation*}
\end{corollary}

\noindent We now prove Theorem \ref{MY_equality}.

\begin{proof}
Note that away from singular fibres, $\omega_{WP} = 0$ if and only if $f$ is a holomorphic fiber bundle, (cf. \cite[Prop. 5.7.]{T19}). If $\omega_{WP} = 0$ then equality holds. If equality holds then,
\begin{align*}
    0 &= (2(n+1) c_2(X) - n c_1(X)^2) \cdot (-2\pi c_1(X))^{n-2}\\
    &= \frac{12(n+1)}{\pi}\int_{X_{can}} \omega_{WP} \wedge \eta_{FS}^{n-2} \geq 0
\end{align*}
since $\omega_{WP}$ is semi-positive. Hence, $\omega_{WP} = 0$.
\end{proof}

\noindent For arbitrary Kodaira dimension, we have the following extension.

\begin{theorem}
Let $X$ be any compact K\"ahler manifold with semi-ample $K_X$, smooth canonical model, $X_{can}$, and $0 < \Kod(X) \leq n-2$. If $[\omega_0]$ is rational, or equivalently, $[\omega_0] = 2 \pi c_1(L_0)$ for some ample $\Q$-line bundle $L_0$, then
\begin{align*}
    (2(n+1) c_2(X) - n c_1(X)^2) &\cdot (2 \pi c_1(L_0))^{n-\kappa-2} \cdot (-2 \pi c_1(X))^{\kappa}\\
    &=  24 (n+1) (2\pi )^{n-\kappa-2}(n-\kappa-2)! r_{n-\kappa-2} [\eta_{FS}]^\kappa,
\end{align*}
where $r_{n-\kappa-2}$ is the coefficient of $j^{n-\kappa-2}$ in the polynomial $\rank(f_*(K_{X/X_{can}}\otimes L_0^j))$ for sufficiently large $j\geq 1$.
\end{theorem}

\begin{proof}
Apply the relative Grothendieck-Riemann-Roch theorem to $L_0^{-j}$ where $j \geq 1$ is a positive integer and $L_0^{-1} = L_0^*$, that is,
\begin{equation*}
    \ch(R f_* L_0^{-j}) = f_* (\ch(L_0^{-j}) \td(\T_f)).
\end{equation*}
Since the fibres are $(n-\kappa)$-dimensional
\begin{align*}
    R f_* L_0^{-j} &= \sum_{i = 0}^{n-\kappa} (-1)^i R^i f_* L_0^{-j}
\end{align*}
and by relative Serre duality, $K_{X/X_{can}} = K_X \otimes f^* K_{X_{can}}^{-1}$, the projection formula, and the relative Kodaira vanishing theorem for derived functors, \cite[Cor. 1.3.]{F26},
\begin{align*}
    (R^i f_* L_0^{-j})^* &= R^{n-\kappa-i} f_*( K_{X/X_{can}} \otimes L_0^{j})\\
    &= R^{n-\kappa-i} f_*(K_X \otimes L_0^j)\otimes K_{X_{can}}^{-1} = 0
\end{align*}
for all $n-\kappa > i$, and sufficiently large $j \geq 1$. It follows that 
\begin{equation*}
    R f_* L_0^{-j} = (-1)^{n-\kappa} R^{n-\kappa} f_* L_0^{-j}
\end{equation*}
and
\begin{align*}
    \ch_0(R f_* L_0^{-j}) &= (-1)^{n-\kappa} \ch_0 (R^{n-\kappa} f_* L_0^{-j})\\
    &= (-1)^{n-\kappa}\rank(R^{n-\kappa} f_* L_0^{-j}) = (-1)^{n-\kappa} \rank(f_*(K_{X/X_{can}}\otimes L_0^j)).
\end{align*}
On the other hand, since
\begin{equation*}
    \ch(L_0^{-j}) = 1 - j c_1(L_0) + \frac{(-j)^2}{2!} c_1(L_0)^2 + \dots + \frac{(-j)^n}{n!} c_1(L_0)^n
\end{equation*}
and the fibres are $(n-\kappa)$-dimensional, then
\begin{equation*}
    f_* ((\ch(L_0^{-j}) \td(\T_f))_{n-\kappa}) = (-1)^{n-\kappa} \rank(f_*(K_{X/X_{can}}\otimes L_0^j))
\end{equation*}
and, in particular, the left-hand side is a polynomial in $j$, for sufficiently large $j \geq 1$. Thus the right-hand side is as well and let us denote the coefficient of $j^i$ in $\rank( f_*(K_{X/X_{can}}\otimes L_0^j))$ by $r_i$. Equating the coefficients of $j^{n-\kappa-2}$ yields
\begin{equation*}
    r_{n-\kappa-2} = \frac{1}{(n-\kappa-2)!}f_*\left(\td_2(\T_f) \wedge c_1(L_0)^{n-\kappa-2}\right).
\end{equation*}
Concerning the Miyaoka-Yau quantity, we can argue in a similar manner to the $\Kod(X) = n-1$ case to find
\begin{align*}
    \int_X (2(n+1) c_2(X) &- n c_1(X)^{2} )\wedge (2 \pi c_1(L_0))^{n-\kappa - 2} \wedge (-2\pi c_1(X))^{\kappa}\\
    &= \int_X 2(n+1) c_2(X) \wedge (2 \pi c_1(L_0))^{n-\kappa-2} \wedge f^*\eta_{FS}^\kappa\\
    &= \int_X 2(n+1)( c_2(\T_f) + c_1(\T_f)^2) \wedge (2 \pi c_1(L_0))^{n-\kappa-2} \wedge f^*\eta_{FS}^\kappa\\
    &= \int_{X_{can}}  24 (n+1)(2\pi)^{n-\kappa-2} f_*\left(\td_2(\T_f) \wedge c_1(L_0)^{n-\kappa-2}\right) \eta_{FS}^\kappa\\
    &= \int_{X_{can}} 24 (n+1) (2\pi )^{n-\kappa-2}(n-\kappa-2)! r_{n-\kappa-2} \eta_{FS}^\kappa.
\end{align*}
\end{proof}

\noindent We ask the natural question.

\begin{question}
For any ample line bundle, $L_0$, is $r_{n-\kappa-2} = 0$ if and only if $f: X \to X_{can}$ is a holomorphic fiber bundle away from singular fibres?
\end{question}

\bibliographystyle{amsplain}

\end{document}